\documentclass[12 pt]{article}%
\usepackage{amsmath, amsfonts, amsthm, color,latexsym}
\usepackage{amsmath,amsthm,amsfonts,mathtools, ulem}
\usepackage{amssymb,amstext,amscd}
\usepackage{latexsym,bm,mathrsfs,esint}
\usepackage{amsmath, ulem}
\usepackage{amsfonts}
\usepackage{amssymb}
\usepackage{color, soul}
\usepackage[all]{xy}
\usepackage{graphicx}%
\providecommand{\U}[1]{\protect\rule{.1in}{.1in}}
\allowdisplaybreaks[4]
\newtheorem{theorem}{Theorem}[section]
\newtheorem{proposition}[theorem]{Proposition}
\newtheorem{corollary}[theorem]{Corollary}
\newtheorem{example}[theorem]{Example}

\newtheorem{remark}[theorem]{Remark}

\newtheorem{lemma}[theorem]{Lemma}
\newtheorem{final remark}[theorem]{Final Remark}
\newtheorem{definition}[theorem]{Definition}
\allowdisplaybreaks[4]

\newcommand {\R}{\mathbb{R}}

\newcommand {\N} {\mathbb{N}}

\begin{document}

\title{$L$- and $M$-weakly compact multilinear operators and their linear adjoints}
\author{Geraldo Botelho\thanks{Supported by FAPEMIG grants RED-00133-21 and APQ-01853-23.}~ and Ariel Mon\c c\~ao\thanks{Supported by a CAPES scholarship. \newline 2020 Mathematics Subject Classification: 46A40, 46B42, 47B07, 47B65, 47H60.\newline Keywords: Riesz spaces, Banach lattices, multilinear operators, linear adjoints, order bounded  variation, $L$-weakly compact operators, $M$-weakly compact operators.}}
\date{}
\maketitle

\begin{abstract} Let $X_1, \ldots, X_m$ be Banach spaces and let $E_1, \ldots, E_m,F$ be Banach lattices. Our main results read as follows: (i) The linear adjoint $A^*$ of a continuous multilinear operator $A \colon X_1 \times \cdots \times X_m \to F$ is $M$-weakly compact if and only if $A$ is $L$-weakly compact. (ii) The linear adjoint $A^*$ of a multilinear operator of order bounded variation $A \colon E_1 \times \cdots \times E_m \to F$ is $L$-weakly compact if and only if the linearization of $A$ on the positive projective tensor product is $M$-weakly compact. In our way to prove these results, we develop the basic theory of linear adjoints of multilinear operators between Riesz spaces, we prove that multilinear operators of order bounded variation between Banach lattices are continuous, and we explore different notions of multilinear operators of $M$-weakly compact-type.
\end{abstract}

\section{Introduction}
For a Banach space $X$ and a Banach lattice $E$, it is well known that the adjoint $T^*$ of a bounded linear operator $T \colon E \to X$ is $L$-weakly compact if and only if $T$ is $M$-weakly compact; and the adjoint $T^*$ of a bounded linear operator $T \colon X \to E$ is $M$-weakly compact if and only if $T$ is $L$-weakly compact \cite[Proposition 3.6.11]{MeyerPeter}. The purpose of this paper is to investigate possible extensions of these dualities to the context of multilinear operators. The main results we obtained are stated in the Abstract.

Multilinear operators between Riesz spaces have been studied since the seminal works of Fremlin \cite{FremlinTPAVL,FremlinTPBL}. We mention important developments due to Buskes and van Rooij \cite{Buskes2}, Bu and Buskes \cite{Buskes3}, Loane \cite{Loane} and Bu, Buskes and Kusraev \cite{Buskes}. Linear adjoints of continuous multilinear operators between Banach spaces have been investigated since the seminal work of Aron and Schottenloher \cite{Aron}, further developments can be found, e.g., in \cite{leodanmathscand,  erhanpilar, silviapablo, Mujicatrans, Ryanpacific, pablo}. The bridge connecting these two trends is the study of linear adjoints of multilinear operators from a Banach space to a Banach lattice, from a Banach lattice to a Banach space and between Banach lattices. It is the intention of this paper to give a contribution to this study.

As to operators between Riesz spaces or Banach lattices, we quickly realized that the correct scenario is the environment of multilinear operators of order bounded variation introduced in \cite{Buskes2}. In Section 2 we develop the basic theory of adjoints of multilinear operators of order bounded variation between Riesz spaces.

For Riesz spaces $E_1, \ldots, E_m,F$, we prove that the adjoint of an $m$-linear operator of order bounded variation $A \colon E_1 \times \cdots \times E_m \to F$ is an order bounded, hence regular, and order continuous linear operator from the order dual of $F$ to the space of regular $m$-linear forms on $E_1 \times \cdots \times E_m$ (cf. Theorem \ref{Theo.1.73M}). In Section 3 we prove that multilinear operators of order bounded variation from a Banach lattice to a normed Riesz space are continuous (cf. Theorem \ref{obv-cont}).

The notion of $L$-weakly compact linear operator can be generalized to multilinear operators in two obvious ways: transposing the linear definition word-by-word and considering multilinear operators that are $L$-weakly compact in each variable, that is, operators that are separately $L$-weakly compact. We compare these two notions to each other and to the classical notion of weakly compact multilinear operator. For Banach spaces $X_1, \ldots, X_m$ and a Banach lattice $E$, we prove in Theorem \ref{Lfc} that the adjoint $A^*$ of a continuous $m$-linear operator $A \colon X_1 \times \cdots \times X_m \to E$ is $M$-weakly compact if and only if $A$ is $L$-weakly compact if and only if the linearization $A_L$ of $A$ on the projective tensor product of $X_1, \ldots,X_m$ is $L$-weakly compact.

In the search of multilinear operators whose adjoints are $L$-weakly compact, a natural step is to explore multilinear extensions of the notion of $M$-weakly compact operator. Looking at the literature, some options are available besides of the natural extension considering disjointness in the cartesian product associated to the coordinatewise ordering: (i) The concept of sporadic disjointness in the cartesian product proved to be useful in \cite{Kusraeva, Kusraeva2}. (ii) In \cite{vinger} one can find applications of the concept of strongly $M$-weakly compact multilinear operators. (iii) Bearing in mind the main result of Section 4, multilinear operators whose linearizations are $L$-weakly compact should also be considered. All three notions, namely, sporadic disjointness, strong $M$-weakly compactness and linearization (which will all be defined in due time) are concepts of a typical multilinear/nonlinear character, which show that they should be taken into account in our search. The result is that we have several classes of multilinear operators of $M$-weakly compact-type to work with. In Section 5 we handle all these classes and, in  Theorem \ref{Teo.M-frac.compBV}, we characterize the multilinear operators $A$ whose adjoints $A^*$ are $L$-weakly commpac and we establish the relations of such operators with the other types of $M$-weakly compact multilinear operators.

All linear spaces are supposed to be real and all ordered linear spaces are supposed to be Archimedean. If $X, X_1, \ldots, X_m$ are Banach spaces, $E, E_1, \ldots, E_m$ are Riesz spaces, then $B_X$ denotes  the closed unit ball of $X$, $E^+$ stands for the positive cone of $E$, in $X_1 \times \cdots \times X_m$ we consider the maximum norm $\|\cdot\|_\infty$ and we consider $E_1 \times \cdots \times E_n$ endowed with the coordinatewise ordering. Spaces of regular linear/multilinear operators are endowed with their corresponding regular norms. For (spaces of) continuous multilinear operators between Banach spaces we refer to \cite{mujica, dineen}, and for Riesz spaces, Banach lattices and linear operators between them we refer to \cite{invitation, Alip, MeyerPeter, Schaefer}.

\section{Adjoints of multilinear operators of order bounded variation}

For linear spaces $E_1, \ldots, E_m,F$, by $L(E_1, \ldots, E_m;F)$ we denote the linear space of all $n$-linear operators from $E_1 \times \cdots \times E_m$ to $F$. In the linear case we simply write $F'$ for the algebraic dual $L(F;\R)$ of $F$.   Following the classical approach due to Aron and Schottenloher \cite{Aron}, the (algebraic) adjoint of $A$ is defined by
$$A^*\colon F'\to L(E_1,\dots,E_m;\mathbb{R})~,~ A^*(y^*)(x_1,\dots,x_m)=y^*(A(x_1,\dots,x_m)),$$
which is a linear operator between linear spaces.

If the underlying spaces are normed spaces, then we denote by ${\cal L}(E_1, \ldots, E_m;F)$ the subspace of $L(E_1, \ldots, E_m;F)$ consisting of continuous operators and by $F^*$ the topological dual of $F$. The (topological) adjoint of a continuous operator $A \in {\cal L}(E_1, \ldots, E_m;F)$ is the restriction of the algebraic adjoint to $F^*$. In this case we have that
$$A^*\colon F^*\to \mathcal{L}(E_1,\dots,E_m;\mathbb{R})$$
is a bounded linear operator between normed spaces. 

Now we consider the case where the underlying spaces are real Riesz spaces. Denoting by $E^\sim$ the order dual of the Riesz space $E$, it is well known  that the order adjoint $T^\sim \colon F^\sim \to E^\sim$ of an order bounded linear operator $T \colon E \longrightarrow F$, defined in the obvious way, is itself an order bounded linear operator \cite[Theorem 1.73]{Alip}. In the multilinear case, for an $m$-linear $A \colon E_1 \times \cdots \times E_m \to F$ between Riesz spaces, the natural step is to define the (order) adjoint of $A$ by the formula
$$A^*(y^*)(x_1,\dots,x_m)=y^*(A(x_1,\dots,x_m)) \mbox{ for all } y^* \in F^\sim, x_1 \in E_1, \ldots, x_n \in E_m.  $$
To be coherent with the previous cases, supposing that $A$ belongs to some Riesz space of multilinear operators from $E_1 \times \cdots \times E_m$, we wish $A^*$ to be an order bounded linear operator taking values in some Riesz space of multilinear forms on $E_1 \times \cdots \times E_m$. To describe the Riesz spaces of multilinear operators/forms that fit this purpose, we need to recall some definitions and fix some notation.

Let $A \colon E_1 \times \cdots \times E_m \to F$ be an $m$-linear operator between Riesz spaces. $A$ is {\it positive} if $A(x_1, \ldots, x_n) \geq 0$ for all $0\leq x_1 \in E_1, \ldots, 0 \leq x_m \in E_m$. In this case we write $A \geq 0$. The space $L(E_1, \ldots, E_n;F)$ of $n$-linear operators from $E_1 \times \cdots \times E_m$ to $F$ is an ordered linear space with the partial order $A \leq B \Longleftrightarrow B - A \geq 0$.

$A$ is {\rm regular}, in symbols $A \in L_r(E_1, \ldots, E_n;F)$, if $A$ can be written as the difference of two positive operators. If $F$ is Dedekind complete, then $L_r(E_1, \ldots, E_n;F)$ is a Dedekind  complete Riesz space \cite[Lemma 2.12 ]{Loane}.

According to \cite{Boulabiar2, Buskes2}, $A$ has  {\it order bounded variation} if, regardless of the positive vectors $a_1 \in E_1, \ldots, a_m \in E_m$, the set consisting of the vectors
$$\sum_{i_1, \ldots, i_m = 1}^{N_1, \ldots, N_m} |A(x^1_{i_1}, \ldots, x^m_{i_m})|,$$
where $N_1, \ldots, N_m \in \mathbb{N}$, $0 \leq x^1_{i_1} \in E_1$ for $i_1 = 1, \ldots, N_1, \ldots, 0 \leq x^m_{i_m} \in E_m$ for $i_m = 1, \ldots, N_m$, $\sum\limits_{i= 1}^{N_1}x^1_i = a_1, \ldots, \sum\limits_{i=1}^{N_n}x^m_i = a_m$, is order bounded in $F$. By $L_{bv}(E_1,\dots,E_m;F)$ we denote the ordered linear subspace of $L(E_1, \ldots, E_n;F)$ consisting of all operators of order bounded variation. A weaker condition on $F$ is required for this space to be a Riesz space: If $F$ is uniformly complete (for  the definition, see \cite[Ex.\,2, p.\,110]{Alip}), then $L_{bv}(E_1,\dots,E_m;F)$ is a Riesz space \cite[Theorem 3.2]{Buskes2}. Moreover, 
if $F$ is Dedekind complete, then
\begin{equation*}L_{bv}(E_1,\dots,E_m;F) = L_{r}(E_1,\dots,E_m;F) \end{equation*}
is a Dedekind complete Riesz space \cite[Theorem 4]{Buskes} and, for any $A \in L_{bv}(E_1,\dots,E_m;F)$ and all $a_1 \in E_1^+, \ldots, a_m \in E_m^+$, it holds
\begin{equation}\label{kbo7}|A|(a_1, \ldots, a_m) = \sup \left\{\sum_{i_1, \ldots, i_m = 1}^{N_1, \ldots, N_m} |A(x^1_{i_1}, \ldots, x^m_{i_m})|\right\},\end{equation}
where the supremum is taken over all $N_1, \ldots, N_m \in \mathbb{N}$, $0 \leq x^1_{i_1} \in E_1$ for $i_1 = 1, \ldots, N_1, \ldots, 0 \leq x^m_{i_m} \in E_m$ for $i_m = 1, \ldots, N_m$, so that $\sum\limits_{i= 1}^{N_1}x^1_i = a_1, \ldots, \sum\limits_{i=1}^{N_n}x^m_i = a_m$ (see \cite[Theorem 3.2(iii)]{Buskes2} and \cite[Proposition 2.14]{Loane}).
In particular, \begin{equation}\label{pl2h}L_{bv}(E_1,\dots,E_m;\mathbb{R}) = L_{r}(E_1,\dots,E_m;\mathbb{R}) \end{equation} is a Dedekind complete Riesz space and the modulus of any regular $m$-linear form can be calculated by the formula (\ref{kbo7}).

   The main purpose of this section is to show that, if $A$ is of order bounded variation, then the order adjoint $A^*$ is a well defined order bounded linear operator from $F^\sim$ to $L_r(E_1, \ldots, E_n)$. First we prove that it is well defined.

\begin{proposition}  Let $A \colon E_1 \times \cdots \times E_m \to F$ be an $m$-linear operator of order bounded variation between Riesz spaces. Then $A^*(y^*)$ is a regular $m$-linear form for every functional $y^* \in F^\sim$. Hence,  $A^*\colon F^\sim\to L_r(E_1,\dots,E_m;\mathbb{R})$ is a well defined linear operator.
\end{proposition}

\begin{proof} Let $y^* \in F^\sim$ be given. It is obvious that $A^*(y^*)$ is an $m$-linear operator, that is, $A^*(y^*) \in L(E_1,\dots,E_m;\mathbb{R})$. To prove that it is regular, by (\ref{pl2h}) it is enough to show that it has order bounded variation.  To do so, let $0 \leq a_1 \in E_1, \ldots, 0\leq a_m \in E_m$ be given. We have to show that the set 
\begin{align*}
D^{y^*}_{a_1,\dots,a_m}\coloneqq\left\{\sum_{i_1, \ldots, i_m = 1}^{N_1, \ldots, N_m} \right.&|A^*(y^*)(x^1_{i_1}, \ldots, x^m_{i_m})|: N_1, \ldots, N_m \in \mathbb{N}, 0\leq x^j_{i_j} \in E_j \mbox{ for } i_j \in \{1, \ldots, N_j\}\nonumber
\\& \left.  \mbox{ and }  j \in\{1,\ldots,m\},\, \sum_{i= 1}^{N_1}x^1_i = a_1, \ldots, \sum_{i=1}^{N_n}x^m_i = a_m\right\}
\end{align*}
is order bounded. As $A$ is of order bounded variation, the set 
\begin{align} D_{a_1,\dots,a_m}\coloneqq \left\{\sum_{i_1, \ldots, i_m = 1}^{N_1, \ldots, N_m} \right.&|A(x^1_{i_1}, \ldots, x^m_{i_m})|: N_1, \ldots, N_m \in \mathbb{N}, 0\leq x^j_{i_j} \in E_j \mbox{ for } i_j \in \{1, \ldots, N_j\}\nonumber
\\& \left.  \mbox{ and }  j \in\{1,\ldots,m\},\, \sum_{i= 1}^{N_1}x^1_i = a_1, \ldots, \sum_{i=1}^{N_n}x^m_i = a_m\right\} \label{lm2x}
\end{align}
is order bounded in $F$. Let $u \in F$ be an upper bound of $D_{a_1,\dots,a_m}$. For all $N_1, \ldots,N_m$, $x^1_1,\ldots, x^1_{N_1},\ldots,$ $ x^m_1,\ldots, x^m_{N_m}$ as in definition of $D^{y^*}_{a_1,\dots,a_m}$,  we have
\begin{align*}
\sum_{i_1, \ldots, i_m = 1}^{N_1, \ldots, N_m} |A^*(y^*)(x^1_{i_1}, \ldots, x^m_{i_m})|&=\sum_{i_1, \ldots, i_m = 1}^{N_1, \ldots, N_m} |y^*(A(x^1_{i_1}, \ldots, x^m_{i_m}))|\\&\leq\sum_{i_1, \ldots, i_m = 1}^{N_1, \ldots, N_m} |y^*|(|A(x^1_{i_1}, \ldots, x^m_{i_m})|)\\&=|y^*|\left(\sum_{i_1, \ldots, i_m = 1}^{N_1, \ldots, N_m} |A(x^1_{i_1}, \ldots, x^m_{i_m})|\right)\leq|y^*|(u).
\end{align*}
So, $|y^*|(u)$ is an upper bound for $D^{y^*}_{a_1,\dots,a_m}$, proving that $A^*(y^*) \in L_r(E_1,\dots,E_m;\mathbb{R})$. The linearity of $A^*$ is immediate.
\end{proof}

Now we need a (positive) multilinear version of \cite[Theorem 1.19]{Alip}.

\begin{lemma}\label{Theo.1.19M.Alip} Let $E_1, \ldots, E_n,F$ be Riesz spaces with $F$ Dedekind complete and let $D$ be a subset of $[L_r(E_1, \ldots,E_n;F)]^+$ satisfying $D \uparrow$.  Then $\sup D$ exists in $L_{r}(E_1,\dots,E_m;F)$ if and only if the set  $\{B(x_1,\dots,x_m):B\in D\}$ is bounded above in $F$ for all $0\leq x_1\in E_1,\dots,0\leq x_m\in E_m$. In this case, $$[\sup D](x_1,\dots,x_m)=\sup\{B(x_1,\dots,x_m):B\in D\}$$ for all $0\leq x_1\in E_1,\dots,0\leq x_m\in E_m$.
\end{lemma}

\begin{proof} Suppose first that $S:=\sup D$ exists in $L_{r}(E_1,\dots,E_m;F)$. In this case, $B(x_1,\dots,x_m)\leq S(x_1,\dots,x_m)$ for all $0\leq x_1\in E_1,\dots,0\leq x_m\in E_m$, hence the set $\{B(x_1,\dots,x_m):B\in D\}$ is bounded above in $F$. 

Assume now that 
 the set $\{B(x_1,\dots,x_m):B\in D\}$ is bounded above in $F$ for all $0\leq x_1\in E_1,\dots,0\leq x_m\in E_m$. The operators in $D$ are all positive, so this set is contained in $F^+$.  As $F$ is Dedekind complete, the map $S\colon {E_1}^+\times\cdots\times {E_m}^+\longrightarrow F^+$ given  by $$S(x_1,\dots,x_m)=\sup\{B(x_1,\dots,x_m):B\in D\}$$
  is well defined. Let us see that $S$ is additive in the first variable. Fix 
  $a_2\in {E_2}^+,\dots,a_m\in {E_m}^+$, and consider the map $$S_1\colon {E_1}^+\longrightarrow F^+~,~ S_1(x_1)=S(x_1,a_2,\dots,a_m).$$
  For any $B\in D$ and all $x_1,y_1\in {E_1}^+$,  $$B(x_1+y_1,a_2,\dots,a_m)=B(x_1,a_2,\dots,a_m)+B(y_1,a_2,\dots,a_m)\leq S_1(x_1)+S_1(y_1).$$
  So, $S_1(x_1+y_1)\leq S_1(x_1)+S_1(y_1)$. Given $B,C\in D$, since $D \uparrow$ there exists $T\in D$ such that $B\vee C\leq T$. Thus, for all $x_1,y_1\in {E_1}^+$,
\begin{align*}
B(x_1,a_2,\dots,a_m)+C(y_1,a_2,\dots,a_m)&\leq T(x_1,a_2,\dots,a_m)+T(y_1,a_2,\dots,a_m)\\&=T(x_1+y_1,a_2,\dots,a_m)\leq S_1(x_1+y_1).
\end{align*}
It follows that $S_1(x_1)+S_1(y_1)\leq S_1(x_1+y_1)$, hence $S_1(x_1+y_1)=S_1(x_1)+S_1(y_1)$. This proves that $S$ is additive in the first variable. The same argument gives additivity in the other variables. By Kantorovich's multilinear extension theorem \cite[Theorem 2.3]{Loane}, $S$ admits a unique positive $m$-linear extension $\overline{S}\colon  E_1\times\cdots\times E_m\longrightarrow F$. All that is left to prove is that $\overline{S}=\sup D$. Of course, $B\leq\overline{S}$ for every $B\in D$. Let $U\in L_{r}(E_1,\dots,E_m;F)$ be such that $B\leq U 
$ for every $B\in D$. Then
$$ \overline{S}(x_1,\dots,x_m) =S(x_1,\dots,x_m)=\sup\{B(x_1,\dots,x_m):B\in D\}\leq U(x_1,\dots,x_m)$$
for all $0\leq x_1\in E_1,\dots,0\leq x_m\in E_m$, which proves that  $\overline{S}=\sup D$. The second assertion is now obvious.
\end{proof}

Recall that a net $(x_\alpha)_{\alpha \in \Gamma}$ in a Riesz space $E$ is {\it order convergent} to  $0$, in symbols $x_\alpha\stackrel{o}{\longrightarrow}0$, if there is a  net $(y_\alpha)_{\alpha \in \Gamma}$ in $E$ such that $y_\alpha \downarrow 0$ and   $|x_\alpha|\leq y_\alpha$ for every $\alpha \in \Gamma$. Furthermore, a linear operator $T\colon E\longrightarrow F$ between Riesz spaces is said to be {\it order continuous} if $T(x_\alpha)\stackrel{o}{\longrightarrow}0$ in $F$ whenever $x_\alpha\stackrel{o}{\longrightarrow}0$ in $E$.


It is worth mentioning that a positive linear operator $T\colon E\longrightarrow F$ between Riesz spaces is order continuous if and only if $T(x_\alpha)\downarrow0$ in $F$ whenever $x_\alpha\downarrow0$ in $E$ \cite[p.\,46]{Alip}. Moreover, order continuous operators are order bounded \cite[Lemma 1.54]{Alip}, but the converse does not hold \cite[Example 1.55]{Alip}.

Next we accomplish the main purpose of this section by proving a multilinear version of \cite[Theorem 1.73]{Alip}.

\begin{theorem}\label{Theo.1.73M} Let $A \colon E_1 \times \cdots \times E_m \to F$ be an $m$-linear operator of order bounded variation between Riesz spaces. Then $A^*\colon F^\sim\to L_r(E_1,\dots,E_m;\mathbb{R})$ is an order bounded, hence regular, and order continuous linear operator. In particular, $|A^*|$ exists in $L_r(F^\sim; L_r(E_1,\dots,E_m;\mathbb{R}))$.
\end{theorem}

\begin{proof} To show that $A^*$ order bounded, let $0\leq y^*\in F^\sim$ be given and consider the set
$$D\coloneqq\left\{\sum_{i = 1}^{N}|A^*(y^*_i)|: N \in \mathbb{N},\, y^*_1, \ldots, y^*_N \in (F^\sim)^+,\, \sum_{i= 1}^{N}y^*_{i} = y^*\right\}.$$
Let us see that $D\uparrow$ in $L_{bv}(E_1,\dots,E_m;F)$. Let  $y^*_1,\dots,y^*_N,z^*_1,\dots,z^*_M$ be positive functionals in $F^\sim$ such that $$\sum_{i=1}^N y^*_i=\sum_{j=1}^M z^*_j=y^*.$$ From the Riesz decomposition property \cite[Theorem 1.20]{Alip} there exist $0\leq w^*_{ij}\in F^\sim$ with $i=1,\dots,N$ and $j=1,\dots,M$, such that $y^*_i=\sum\limits_{j=1}^M w^*_{ij}$ and $z^*_j=\sum\limits_{i=1}^N w^*_{ij}$ for all $i$ and $j$. It follows that $\sum\limits_{i,j=1}^{N,M} w_{ij}=y^*$. We have
   $$\sum_{i = 1}^{N} |A^*(y^*_{i})|=\sum_{i = 1}^{N}\left|\sum_{j = 1}^{M} A^*(w^*_{ij})\right|\leq\sum_{i,j=1}^{N,M} |A^*(w_{ij})|$$
and, analogously, $\sum\limits_{j = 1}^{M} |A^*(z^*_{i})|\leq\sum\limits_{i,j=1}^{N,M} |A^*(w_{ij})|$, showing that $D\uparrow$.

Next we shall check that $S:=\sup D$ exists in $L_{bv}(E_1,\dots,E_m;\R)$. To accomplish this task, fix $0\leq a_1\in E_1,\dots,0\leq a_m\in E_m$. As $A$ is of order bounded variation, there exists $0\leq u\in F$ so that the set $D_{a_1, \ldots, a_m}$ defined in (\ref{lm2x}) is contained in 
the order interval $[0,u]$. We shall say that $(\Delta)$ holds whenever $N_1, \ldots, N_m \in \mathbb{N}$, $0 \leq x^1_{i_1} \in E_1$ for $i_1 = 1, \ldots, N_1, \ldots, 0 \leq x^m_{i_m} \in E_m$ for $i_m = 1, \ldots, N_m$, are so that $\sum\limits_{i= 1}^{N_1}x^1_i = a_1, \ldots, \sum\limits_{i=1}^{N_n}x^m_i = a_m$.  If $y^*_1,\dots,y^*_N$ are positive functionals in $F^\sim$ with $\sum\limits_{k=1}^{N} y^*_k=y^*$, then
\begin{align*}
\left[\sum_{k = 1}^{N} |A^*(y^*_{k})|\right](a_1,\dots,a_m)&=\sum_{k = 1}^{N} |A^*(y^*_{k})|(a_1,\dots,a_m)\\
&\stackrel{\rm (\ref{kbo7})}{=}\sum_{i = 1}^{N}\sup\left\{\sum_{i_1, \ldots, i_m = 1}^{N_1, \ldots, N_m} |A^*(y^*_k)(x^i_1, \ldots, x^i_m)|: (\Delta) \mbox{ holds}\right\}
 \\&=\sum_{i = 1}^{N}\sup\left\{\sum_{i_1, \ldots, i_m = 1}^{N_1, \ldots, N_m} |y^*_k(A(x^i_1, \ldots, x^i_m))|: (\Delta) \mbox{ holds}\right\} 
 \\&\leq\sum_{i = 1}^{N}\sup\left\{y^*_k\left(\sum_{i_1, \ldots, i_m = 1}^{N_1, \ldots, N_m} |A(x_i^1, \ldots, x_i^m|  \right): (\Delta) \mbox{ holds}\right\} 
 \\&\leq \sum_{i = 1}^{N}y^*_k(u)\leq y^*(u).
\end{align*}
This shows that the set $\{B(a_1,\dots,a_m):B\in D\}$ is bounded in $\R$ for all $0\leq a_1\in E_1,\dots,0\leq a_m\in E_m$. As $L_{bv}(E_1,\dots,E_m;\R)$ is Dedekind complete, the supremum $S:=\sup D$ exists in $L_{bv}(E_1,\dots,E_m;\R)$ by Lemma \ref{Theo.1.19M.Alip}.

For any $0\leq z^*\leq y^*$, the sum $|A^*(z^*)|+|A^*(y^*-z^*)|$ is an element of $D$, therefore  $$|A^*(z^*)|\leq|A^*(z^*)|+|A^*(y^*-z^*)|\leq S.$$ This proves that $A^*([0,y^*])\subset[-S,S]$, hence $A^*$ is order bounded. As $L_r(E_1, \ldots, E_m;\mathbb{R})$ is Dedekind complete, $A^*$ is regular. $|A^*|$ exists since $A^*$ is an order bounded operator in $L_r(F^\sim; L_r(E_1,\dots,E_m;\mathbb{R}))$, which is a Riesz space because $L_r(E_1,\dots,E_m;\mathbb{R})$ is Dedekind complete.

     To check that $A^*$ is order continuous, let $0\leq a_1\in E_1,\dots,0\leq a_m\in E_m$ be fixed. Again, let $0\leq u\in F$ be an upper bound of the set $D_{a_1,\ldots, a_m}$ defined in (\ref{lm2x}). 
     Given $0 \leq y^* \in F^\sim$, above we proved that 
     \begin{equation}\label{ol2v}\left[\sum_{k = 1}^{N} |A^*(y^*_{k})|\right](a_1,\dots,a_m)\leq y^*(u)\end{equation}
     for every finite family  $y^*_1,\dots,y^*_N$ of positive linear functionals in $F^\sim$ with $\sum\limits_{k=1}^{N} y^*_k=y^*$. It follows from \cite[Theorem 1.21]{Alip} that
     $$\left\{\sum_{k = 1}^{N} |A^*(y^*_{k})|: N\in \mathbb{N}, 0\leq y_k^* \in F^\sim, \sum_{k^=1}^n y_k^* = y^* \right\} \uparrow |A^*|(y^*).  $$
By Lemma \ref{Theo.1.19M.Alip},
$$|A^*|(y^*)(a_1, \ldots, a_m) = \sup\left\{\left[\sum_{k = 1}^{N} |A^*(y^*_{k})|\right](a_1,\dots,a_m):N\in \mathbb{N}, 0\leq y_k^* \in F^\sim, \sum_{k^=1}^n y_k^* = y^*   \right\}. $$
Calling on (\ref{ol2v}) it follows that $|A^*|(y^*)(a_1,\dots,a_m)\leq y^*(u)$ for each $0\leq y^*\in F^\sim$. In particular, for every net $(y_\alpha^*)_\alpha$ in $F^\sim$ with $y^*_\alpha\downarrow0$,
     $$|A^*|(y^*_\alpha)(a_1,\dots,a_m)\leq y^*_\alpha(u)\downarrow0.$$
     It follows that $|A^*|(y^*_\alpha)(a_1,\dots,a_m)\downarrow0$ in $\R$ whenever $y^*_\alpha\downarrow0$ in $F^\sim$. Since this holds for all  $0\leq a_1\in E_1,\dots,0\leq a_m\in E_m$, we have $|A^*|(y^*_\alpha)\downarrow0$  in $L_{bv}(E_1,\dots,E_m;\R)$ whenever $y^*_\alpha\downarrow0$ in $F^\sim$. This proves that $|A^*|$ is order continuous, therefore  $A^*$ is order continuous as well \cite[Theorem 1.56]{Alip}. 
\end{proof}

For Riesz spaces $E_1, \ldots, E_m, F$ with $F$ Dedekind complete, every multilinear operator $A $ in $L_{bv}(E_1,\dots,E_m;F)$ has an absolute value $|A| \in L_{bv}(E_1,\dots,E_m;F)$ (see \cite{Buskes}). On the one hand, it is easy to check that $|A^*|\leq |A|^*$; on the other hand, equality does not hold even in the linear case (see \cite[Example 3.2]{khazhak}). We shall finish this section showing that $|A^*|$ and $|A|^*$ coincide on order continuous functionals in $F^\sim$. First we establish an inequality that does not require $F$ to be Dedekind complete.

\begin{lemma} \label{lemma1.75M}
Let $E_1, \ldots, E_m, F$ be Riesz spaces. For every $A \in L_{bv}(E_1,\dots,E_m;F)$ and all $0\leq y^*\in F^\sim$, $0\leq x_1\in E_1,\dots,0\leq x_m\in E_m$, it holds $$y^*(|A(x_1,\dots,x_m)|)\leq|A^*|(y^*)(x_1,\dots,x_m).$$
\end{lemma}

\begin{proof} Given $0\leq y^*\in F^\sim$ and $0\leq x_1\in E_1,\dots,0\leq x_m\in E_m$, write $x=(x_1,\dots,x_m)$. By \cite[Theorem 1.23]{Alip} there exists $z^*\in F^\sim$ such that $|z^*|\leq y^*$ and $y^*(|A(x)|)=z^*(A(x))$. Then, $$y^*(|A(x)|)= z^*(A(x)) =A^*(z^*)(x)\leq|A^*|(|z^*|)(x)\leq|A^*|(y^*)(x).$$
\end{proof}

Now we can prove a multilinear version of the Krengel-Synnatzchke Theorem \cite[Theorem 1.76]{Alip}.

\begin{proposition} Let $E_1, \ldots, E_m, F$ be Riesz spaces with $F$ Dedekind complete and let $A \in L_{bv}(E_1,\dots,E_m;F)$ be given. Then, $|A^*|(y^*)=|A|^*(y^*)$ for every order continuous functional $y^*\in F^\sim$.
\end{proposition}

\begin{proof} We shall prove the bilinear case $m=2$, the general case follows the same argument. Let $A \in L_{bv}(E_1,E_2;F)$ be given and let $y^*\in F^\sim$ be an arbitrary order continuous positive functional. As mentioned above, $|A^*|(y^*) \leq |A|^*(y^*)$ holds. To prove the reverse inequality, let $0\leq a\in E_1$ and $0\leq b\in E_2$ be given. We shall say that $(\Gamma)$ holds whenever $N,M \in \mathbb{N}$, $a_1,\ldots, a_N \in E_1^+, b_1, \ldots, b_M \in E_2^+$ are so that $\sum\limits_{i=1}^N a_i = a, \sum\limits_{j=1}^M b_j = b$. We have
\begin{align*}
[A|^*(y^*)(a,b)&=y^*(|A|(a,b))\\
			   &=y^*\left(\sup\left\{ \sum_{i, j = 1}^{N,M}|A(a_{i},b_{j})| : (\Gamma) \mbox{ holds}\right\}\right)\\ 
			   &=\sup\left\{ \sum_{i, j = 1}^{N,M}y^*\left(|A(a_{i},b_{j})|\right) :(\Gamma) \mbox{ holds}\right\}\\ 
			   &\leq\sup\left\{ \sum_{i, j = 1}^{N,M}[|A^*|(y^*)](a_i,b_j) :(\Gamma) \mbox{ holds}\right\}\\
			   &=|A^*|(y^*)(a,b),
\end{align*}
where the second and the last equalities follow from (\ref{kbo7}), the third holds because $y^*$ is order continuous, and the inequality follows from Lemma \ref{lemma1.75M}. Thus, $|A|^*(y^*) \leq |A^*|(y^*)$, hence $|A^*|(y^*) = |A|^*(y^*)$, for every positive order continuous functional $y^*\in F^\sim$. For the case of an arbitrary regular order continuous $y^*\in F^\sim$, note  that both $|y^*|$ and $|y^*|-y^*$ are positive order continuous functionals in $F^\sim$ \cite[Theorem 1.56]{Alip}. From what we have just proved for positive order continuous functionals, it follows that
\begin{align*}
|A^*|(y^*)&=|A^*|(|y^*|-(|y^*|-y^*))=|A^*|(|y^*|)-|A^*|(|y^*|-y^*)\\&=|A|^*(y^*)-|A|^*(|y^*|-y^*)=|A|^*(y^*)-|A|^*(|y^*|)+|A|^*(y^*)=|A|^*(y^*),
\end{align*}
which completes the proof.
\end{proof}

Combining the proposition above with \cite[Theorem 2.26 and Corollary 2.27]{invitation} we obtain the following.

\begin{corollary} Let $E_1, \ldots, E_m$ be Riesz spaces and let $F$ be a Banach lattice with order continuous norm. Then $|A^*|=|A|^*$ for every  $A \in L_{bv}(E_1,\dots,E_m;F)$.  
\end{corollary}

\section{Continuity of operators of order bounded variation}

For further use, we shall prove in this section that multilinear operators of order bounded variation are continuous.

It is well known, and widely used, that order bounded linear operators from a Banach lattice to a normed Riesz space are bounded. For instance, this result is stated in \cite[Theorem 1.31]{invitation}. We shall give two short reasonings for the benefit of the reader:

\begin{lemma}\label{lemaob-cont} If $E$ is a  Banach lattice and $F$ is a normed Riesz space, then every order bounded linear operator from $E$ to $F$ is bounded.
\end{lemma}

\begin{proof} Let $T \colon E \to F$ be an order bounded linear operator.

  First proof. By $F^\delta$ we denote the Dedekind completion of $F$, which is a Dedekind complete Banach lattice containing an isometric lattice copy of $F$ \cite[Ex.\,21, p.\,26]{invitation}. In particular, there is an isometric embedding $i \colon F \to F^\delta$ which is also a Riesz isomorphism onto its range. Since $T$ is order bounded and $i$ is positive, $i \circ T \colon E \to F^\delta$ is an order bounded linear operator taking values in a Dedekind complete Banach lattice. By \cite[Theorem 1.18]{Alip}, $i \circ T$ is regular, hence bounded, therefore there is $C > 0$ so that $\|T(x)\| = \|i(T(x))\| \leq C \|x\|$ for every $ x \in E. $

   Second proof. By \cite[Proposition 3.1]{Wickstead} there is $C > 0$ such that, for every $x\in E^+$, there is  $y\in F^+$ so that $T([-x,x])\subset[-y,y]$ and $\|y\|\leq C\|x\|$. Fix $x\in B_E$. Since $x\in[-|x|,|x|]$ and $\|x\|=\||x|\|$, we have $|x|\in B_E$. Then, there exists $y\in F^+$ so that $T(x)\in T([-|x|,|x|])\subset[-y,y]$ and $\|y\|\leq C\||x|\|= C\|x\|$. As $-y\leq T(x)\leq y$, $|T(x)|\leq y$, hence $\|T(x)\|\leq \|y\|$ because $F$ is a normed Riesz space. Therefore, $\|T(x)\|\leq C\|x\|$.
\end{proof}

\begin{lemma} \label{lemaobv-sob} Let
$A \colon E_1 \times \cdots \times E_m \to F$ be an  $m$-linear operator of order bounded variation between Riesz spaces. Then $A$ is separately order bounded, that is, for all $a_1\in E_1,\dots,a_m\in E_m$ and every $1 \leq i \leq m$, the linear operator
$$x_i \in E_i \mapsto A(a_1,\dots,a_{i-1},x_i,a_{i+1},\dots,a_m) \in F$$ is order bounded. 
\end{lemma}
\begin{proof} Fix $1\leq i\leq m$ and $a_1\in {E_1}^+,\dots,a_m\in {E_m}^+$. As $A$ is of order bounded variation, for every $0 \leq x_i \in E_i$ the set 
\begin{equation} \label{yt4z}\left\{\sum_{j = 1}^{N_i} |A(a_1,\dots,a_{i-1},x^i_j,a_{i+1},\dots,a_m)|:\right. N_i \in \mathbb{N}, 0\leq x^i_j \in E_i, \left.\sum_{j= 1}^{N_i}x^i_j = x_i\right\}
 \end{equation} is order bounded in $F$. Hence, the set
$$\left\{|A(a_1,\dots,a_{i-1},y_{i},a_{i+1},\dots,a_m)|:0\leq y_i\leq x_i\right\}$$ is order bounded in $F$, for every $0 \leq x_i \in E_i$, because every element in this set is smaller than or equal to an element of the set (\ref{yt4z}). Calling $A_i$ the linear operator $$x_i \in E_i\mapsto A(a_1,\dots,a_{i-1},x_i,a_{i+1},\dots,a_m) \in F, $$
it follows that $A_i([0,x_i])$ is order bounded for every $0 \leq x_i \in E_i$, which proves that $A_i$ is order bounded.

Now, fix $a_2\in {E_2},a_3\in {E_3}^+,\dots,a_m\in {E_m}^+$. Using what we have just proved, the linear operators
\begin{align*}
x_1 \in E_1\mapsto A(x_1,{a_2}^+,a_3,\dots,\dots,a_m) \in F\text{ ~and}\\x_1 \in E_1 \mapsto A(x_1,{a_2}^-,a_3,\dots,\dots,a_m)\in F~~~~~~\,
\end{align*} are order bounded. As $A$ is $m$-linear and the sum of order bounded linear operators is order bounded, the operator
\begin{equation}\label{8jxq} x_1 \in E_1\mapsto A(x_1,{a_2},a_3,\dots,\dots,a_m) \in F \end{equation}
 is order bounded as well. Repeating the argument, the operator (\ref{8jxq}) is order bounded for fixed $a_2\in {E_2},a_3\in {E_3}, a_4 \in E_4^+,\dots,a_m\in {E_m}^+$. After finitely many repetitions, we get that the operator (\ref{8jxq}) is order bounded for fixed $a_2\in {E_2},a_3\in {E_3}, \ldots, a_m\in E_m$. The same argument gives the order boundedness  in the other variables.
\end{proof}

\begin{proposition} \label{obv-cont}
 Let $E_1, \ldots, E_m$ be Banach lattices and let $F$ be a normed Riesz space. Then every $m$-linear operator $A \colon E_1 \times \cdots \times E_m \to F$ of order bounded variation is continuous.
\end{proposition}
\begin{proof} By Lemma \ref{lemaobv-sob}, $A$ is separately order bounded, so $A$ is separately continuous by Lemma \ref{lemaob-cont}. The result follows because separately continuous multilinear operators defined on the cartesian product of Banach spaces are continuous (see \cite[Corollary 2.3.5]{Botelho} for the bilinear case or \cite[Theorem 7, p.\,4]{smooth} for the general case). 
\end{proof}

\section{$L$-weakly compact multilinear operators}

Recall that a bounded linear operator $T$ from a Banach lattice $E$ to a Banach space $X$ is $L$-weakly compact if $\|y_n\| \longrightarrow 0$ for every disjoint sequence $(y_n)_n$ in the solid hull of $T(B_E)$ \cite[Definition 5.59]{Alip}.

There are two obvious transpositions of this notion to the multilinear setting. In this section, $X_1,\dots,X_m$ are Banach spaces and $E$ is a Banach lattice. In the cartesian product of normed spaces we shall work with the maximum norm.

\begin{definition}\rm A continuous $m$-linear operator $A\colon X_1\times\cdots\times X_m\to E$ is said to be:\\
(i) {\it $L$-weakly compact} if $\|y_n\|\longrightarrow 0$ whenever $(y_n)_n$ is a disjoint sequence  in the solid hull of $A(B_{X_1 \times \cdots \times X_m})$.\\
(ii) {\it Separately $L$-weakly compact} if $A$ is $L$-weakly compact in each variable, that is, for all $a_1\in X_1,\dots,a_m\in X_m$ and every $i \in \{1, \ldots,m\}$, the bounded linear operator $$x_i \in X_i \mapsto A(a_1,\dots,a_{i-1},x_i,a_{i+1},\dots,a_m)\in X$$
is $L$-weakly compact.
\end{definition}

By ${\cal L}(X_1, \ldots, X_m;E)$ we denote the linear subspace of $L(X_1, \ldots, X_m;E)$ consisting of all continuous operators, which is a Banach space with the usual supremum norm. It is easy to see that the sets of $L$-weakly compact and of separately $M$-weakly compact operators are linear subspaces of ${\cal L}(X_1, \ldots, X_m;E)$. A direct transposition of the proof of \cite[Theorem 5.61]{Alip} shows that $L$-weakly compact multilinear operators are weakly compact (recall that a multilinear operator is weakly compact if bounded sequences are sent to sequences admitting a weakly convergent subsequence). This is no longer true for separately $L$-weakly compact operators:

\begin{example}\label{ol4h}\rm As $\ell_2$ is reflexive, the continuous bilinear operator
$$A\colon \ell_2\times\ell_2\to\ell_1~,~A\left((a_n)_n,(b_n)_n\right)=\left(a_nb_n\right)_n,$$
is separately weakly compact. But $\ell_1$-valued weakly compact linear operators are $L$-weakly compact \cite[Theorem 5.62]{Alip}, therefore $A$ is separately $L$-weakly compact. Since the sequence $(e_n)_n$ of canonical unit vectors has no weakly convergent (or, equivalently, norm convergent) subsequence in $\ell_1$ and $A(e_n,e_n) = e_n$ for every $n$, $A$ fails to be weakly compact.
\end{example}


\begin{proposition}\label{pl4x} $L$-weakly compact multilinear operators are  separately $L$-weakly compact.
\end{proposition}

\begin{proof} Let $A\colon X_1\times\cdots\times X_m\to E$ be an $L$-weakly compact $m$-linear operator and fix $0 \neq a_2\in X_2,\dots,0 \neq a_m\in X_m$. We have to show that the linear operator
$$A_{a_2,\dots,a_m}\colon X_1\longrightarrow E~,~A_{a_2,\dots,a_m}(x_1)=A(x_1,a_2,\dots,a_m),$$
is $L$-weakly compact. Let $(y_n)_n$ be a disjoint sequence in  $A_{a_2,\dots,a_m}(B_{X_1})$. From
\begin{align*}
A_{a_2,\dots,a_m}(B_{X_1})&=A(B_{X_1}\times\{a_2\}\times\cdots\times\{a_m\})\\&\subseteq \|a_2\|\cdots\|a_m\|\!\cdot\! A(B_{X_1}\times B_{X_2}\times\cdots\times B_{X_m}),
\end{align*}
we get that $(y_n)_n$ is a disjoint sequence in $\|a_2\|\cdots\|a_m\|\!\cdot\! A(B_{X_1}\times B_{X_2}\times\cdots\times B_{X_m})$. As $\|a_2\|\cdots\|a_m\|\!\cdot\! A$ is an $L$-weakly compact $m$-linear opertor, $\|y_n\|\to0$, proving that $A_{a_2,\dots,a_m}$ is an $L$-weakly compact linear operator. The cases of the other variables follow similarly.
\end{proof}

The separately $L$-weakly compact bilinear operator from Example \ref{ol4h} is not $L$-weakly compact because it is not weakly compact. On the one hand, this shows that the converse of the proposition above does not hold; on the other hand, this raises the question whether or not multilinear operators that are simultaneously weakly compact and separately $L$-weakly compact are $L$-weakly compact. By making a slight - but efficient - modification in Example \ref{ol4h}, next we show that this is not case. Moreover, in the next example, which shall be useful later again, we describe an interesting feature of disjoint sequences in sequence spaces.

 Recall that an operator $T$ from a Banach lattice $E$ to a Banach space $X$ is $M$-weakly compact if $T(x_n) \longrightarrow 0$ in $F$ for every norm bounded disjoint sequence $(x_n)_n$ in $E$ \cite[Definition 5.59]{Alip}.

\begin{example}\rm \label{cont.Lfc-ex1} 
Let $E$ be a Riesz space whose elements are real-valued sequences and whose ordering is the coordinatewise ordering. This ordering guarantees the following implication:
\begin{equation} \label{uynh}
\mbox{If } (a_n)_n\perp(b_n)_n \mbox{ in } E, \mbox{then } a_nb_n=0 \mbox{ for every } n\in\N.
\end{equation}Given a disjoint sequence $(\alpha^j)_j=((a^j_n)_n)_j$ in $E$, let us prove that there exists an increasing sequence $(m_k)_k$ in $\N$ so that, for every $k\in\N$, $a^j_i=0$ whenever $i \in \{1,\dots,k\}$  and $j>m_k$. Indeed:

For $k=1$, if $a_1^j=0$ for every $j$ then we put $m_1=1$. Otherwise, let  $m_1$ be any natural number such that $a_1^{m_1}\neq0$. By (\ref{uynh}), $a_1^j = 0$ for every $j \neq m_1$, in particular $a_1^j = 0 $ for every $j \geq m_1$. 

For $k=2$, if $a_2^j=0$ for every $j>m_1$ then we put $m_2=m_1+1$. Otherwise, let $m_2$ be any natural number such that $m_2>m_1$ and $a_2^{m_2}\neq0$. By (\ref{uynh}), $a_2^j =0$ for every $j> m_2$. Since $m_2 > m_1$, we also have $a_1^j = 0 $ for $j > m_2$. 

Inductively, suppose that $m_1<\cdots<m_{k-1}$ have been chosen so that $a_{1}^j = \cdots = a_{k-1}^{j} = 0$ for every $j > m_{k-1}$. If $a_k^j=0$ for every $j>m_{k-1}$ then we put $m_k=m_{k-1}+1$. Otherwise, let $m_k$ be any natural number such that $m_k>m_{k-1}$ and $a_k^{m_k}\neq0$. By (\ref{uynh}), $a_k^j=0$ for every $j> m_k$, hence $a_{1}^j = \cdots = a_{k-1}^{j} = a_k^j= 0$ for $j> m_k$.

It was shown in \cite[Example 3]{Torres} that
    $$A\colon \ell_2\times\ell_2\to\ell_2~,~A\left((a_n)_n,(b_n)_n\right)=\left(a_nb_n\right)_n,$$
is a non-compact separately compact bilinear operator. The reflexivity of  the target space $\ell_2$ gives that $A$ is weakly compact. Let us see that it is separately $L$-weakly compact but not $L$-weakly compact.
Let $(\alpha^j)_j=((a^j_n)_n)_j$ be an arbitrary norm bounded disjoint sequence in $\ell_2$. Fixed $b=(b_n)_n\in\ell_2$, let us check that the linear operator  $A_b\colon\ell_2\to\ell_2$, given by $A_b(a)=A(a,b)$, is $M$-weakly compact. Given a disjoint sequence $(\alpha^j)_j=((a^j_n)_n)_j$ in $B_{\ell_2}$, let $(m_k)_k$ be the increasing sequence of natural numbers constructed in the first part of this example. For each $k\in\N$, the following holds for every $j>m_k$:
\begin{align*}
\|A_b(\alpha^j)\|^2_2&=\|(a_1^j b_1,a_2^j b_2,\dots)\|^2_2=\sum_{n=1}^\infty |a_n^j|^2\cdot |b_n|^2=\sum_{n=k+1}^\infty |a_n^j|^2\cdot |b_n|^2\leq \sum_{n=k+1}^\infty |b_n|^2.
\end{align*}
Since $b\in\ell_2$, given $\varepsilon>0$ there exists $n_0$ such that $\sum\limits_{n=n_0}^\infty |b_n|^2<\varepsilon$. And, as the sequence $(m_k)_k$ is increasing, we can take $k_0$ with $m_{k_0}>n_0$. Putting $j_0 = m_{m_{k_0}} \in \mathbb{N}$,  we get
 $$\|A_b(\alpha^j)\|^2_2 \leq \sum_{n=m_{k_0}+1}^\infty |b_n|^2 \leq \sum_{n=m_{k_0}}^\infty |b_n|^2 \leq  \sum_{n=n_0}^\infty |b_n|^2  < \varepsilon \mbox{ for every } j \geq j_0.$$
 This proves that $A_b(\alpha^j)\to0$, hence $A_b$ is $M$-weakly compact. Since $\ell_2 = \ell_2^*$ has order continuous norm, it follows from \cite[Theorem 5.67]{Alip} that $A_b$  is $L$-weakly compact. The symmetry of $A$ gives that $A$ is separately $L$-weakly compact.

To see that $A$ fails to be $L$-weakly compact, note that the sequence $(e_n)_n$ of canonical unit vectors of $\ell_2$ is a normalized - hence not norm null - disjoint sequence contained in (the solid hull) of $A(B_{\ell_2 \times \ell_2})$ because $A(e_n,e_n) = e_n$ for every $n$. 
\end{example}

One side of the duality between $L$- and $M$-weakly compact linear operators establishes that an operator $T \in {\cal L}(X ; E)$ is $L$-weakly compact if and only if $T^* \in {\cal L}(E^* ; X^*)$ is $M$-weakly compact \cite[Theorem 5.64(2)]{Alip}. Our next aim is to prove the multilinear counterpart of this duality. In order to prove a bit more than that, some terminology and notation are needed.

By $X_1 \widehat{\otimes}_\pi \cdots \widehat{\otimes}_\pi X_m$ we denote the completed projective tensor product of the Banach spaces $X_1, \ldots, X_m$ endowed with the projective norm $\pi$. Denoting by
$$
\sigma_m \colon X_1 \times \cdots \times X_m \to X_1 \widehat{\otimes}_\pi \cdots \widehat{\otimes}_\pi X_m, \quad (x_1, \ldots, x_m) \mapsto x_1 \otimes \cdots \otimes x_m,$$
the canonical norm one $m$-linear operator, for each operator $A \in \mathcal{L}(X_1, \ldots, X_m; Y)$ there exists a unique bounded linear operator $A_L\in \mathcal{L}(X_1 \widehat{\otimes}_\pi \cdots \widehat{\otimes}_\pi X_m; Y)$, called the linearization of $A$, such that $A = A_L \circ \sigma_m$. Moreover, the correspondence \begin{equation} \label{re9k}A \in \mathcal{L}(X_1, \ldots, X_m; Y) \mapsto A_L \in \mathcal{L}(X_1 \widehat{\otimes}_\pi \cdots \widehat{\otimes}_\pi X_m; Y)\end{equation}
is an isometric isomorphism. For details, see \cite{df, Ryan}.

It is known that a multilinear operator $A$ is (weakly) compact if and only if its adjoint $A^*$ is a (weakly) compact linear operator if and only if its linearization $A_L$ is (weakly) compact \cite{Aron, prims, Mujicatrans, Ryanpacific}.

Bearing the two paragraphs above in mind, the next result is the best one can expect for $L$-weakly compact multilinear operators.

\begin{theorem}\label{Lfc} Let $X_1,\dots,X_m$ be Banach spaces and let $E$ be a Banach lattice. The following are equivalent for an operator $A\in\mathcal{L}(X_1,\dots,X_m;E)$.\\
{\rm (i)} $A$ is $L$-weakly compact.\\
{\rm (ii)} $A^*\colon E^*\to \mathcal{L}(X_1,\dots,X_m;\mathbb{R})$ is an $M$-weakly compact linear operator. \\
{\rm (iii)} $A_L\colon X_1 \widehat{\otimes}_\pi \cdots \widehat{\otimes}_\pi X_m\to E$ is an $L$-weakly compact linear operator.
\end{theorem}

\begin{proof} (i)$\Leftrightarrow$(ii) By Hahn-Banach, $A$ is $L$-weakly compact if and only if
\begin{equation}\label{oim2}\|y_n\|=\sup\limits_{y^*\in B_{E^*}}|y^*(y_n)|\longrightarrow 0 \mbox{ for every disjoint sequence } (y_n)_n \subseteq {\rm Sol}(A(B_{X_1 \times \cdots \times X_m})).\end{equation}
As $A(B_{X_1 \times \cdots \times X_m})\subseteq E$ and $B_{E^*}\subseteq E^*$, by the Burkinshaw–Dodds Theorem \cite{Burk-Dodds} (see also \cite[Theorem 5.63]{Alip}), (\ref{oim2}) is equivalent to
\begin{equation}\label{pm9w}
\sup\limits_{y\in A(B_{X_1 \times \cdots \times X_m})}|y^*_n(y)|\longrightarrow 0 \mbox{ for every disjoint sequence } (y^*_n)_n \subseteq {\rm Sol}(B_{E^*})=B_{E^*}.\end{equation}
Note that, for each $y^*\in E^*$,
\begin{align*}
\sup\limits_{y\in A(B_{X_1 \times \cdots \times X_m})}|y^*(y)|&=\sup\{|y^*(A(x_1,\dots,x_m))|:(x_1,\dots,x_m)\in B_{X_1 \times \cdots \times X_m}\}\\&=\sup\{|[A^*(y^*)](x_1,\dots,x_m)|:(x_1,\dots,x_m)\in B_{X_1 \times \cdots \times X_m}\}=\|A^*(y^*)\|.
\end{align*}
So, (\ref{pm9w}) holds if and only if
  $\|A^*(y^*_n)\|\longrightarrow 0$ for every disjoint sequence $(y^*_n)_n$ in $B_{E^*}$; that is, if and only if $A^*$ is $M$-weakly compact.

  \medskip

\noindent(ii)$\Leftrightarrow$(iii) For all $x_1\in X_1,\dots,x_m\in X_m$ and $y^*\in E^*$,
\begin{align*}
(A^*(y^*))_L(x_1\otimes\cdots\otimes x_m))&=A^*(y^*)(x_1,\dots,x_m)=y^*(A(x_1,\dots,x_m))\\&= y^*(A_L(x_1\otimes\cdots\otimes x_m))=[(A_L)^*(y^*))](x_1\otimes\cdots\otimes x_m).
\end{align*}
The linearity and the uniqueness of $(A^*(y^*))_L$ gives $(A^*(y^*))_L=(A_L)^*(y^*)$. Recalling that $A^*(y^*)\in\mathcal{L}(X_1,\dots,X_m;\R)$, by (\ref{re9k}) we have
$$\|A^*(y^*)\|= \|(A^*(y^*))_L\|=\|(A_L)^*(y^*)\| $$
for every $y^* \in E^*$. Therefore, $A^*$ is $M$-weakly compact if and only if $(A_L)^*$ is $M$-weakly compact if and only if (by \cite[Theorem 5.64(2)]{Alip}) $A_L$ is $L$-weakly compact.
\end{proof}

%

\section{$M$-weakly compact multilinear operators}

Unless explicitly stated otherwise, in this section $E, E_1, \ldots, E_m,F$ are Banach lattices and $X$ is a Banach space. It is well known that a bounded linear operator $T \colon E \to X$ is $M$-weakly compact if and only if its adjoint $T \colon X^* \to E^*$ is $L$-weakly compact \cite[Theorem 5.64(1)]{Alip}. One (obvious) reason that makes this possible is the fact that $E^*$ is a Banach lattice. For a continuous $m$-linear operator $A \colon E_1 \times \cdots \times E_m \to X$, the situation is a bit more involved because the adjoint $A^*$ takes values in ${\cal L}(E_1, \ldots, E_m;\mathbb{R})$, which is not a Banach lattice - not even a Riesz space - in general. To overcome this issue, we work with multilinear operators $A \colon E_1 \times \cdots \times E_m \to F$ between Banach lattices of order bounded variation because, in this case, the adjoint $A^*$ of $A$ takes values in $L_r(E_1, \ldots, E_m;\mathbb{R})$, which is a Banach lattice (cf. (\ref{pl2h})).

Since regular linear/multilinear operators are continuous, we shall use the symbol ${\cal L}_r$ instead of $L_r$ to denote the spaces of such operators.

A more sensitive issue is the identification of the class of multilinear operators whose adjoints are $L$-weakly compact. Of course, the first attempts should be the two natural multilinear  counterparts of the class of $M$-weakly compact linear operators which we define next.

\begin{definition}\rm A continuous $m$-linear operator $A \colon E_1 \times \cdots \times E_m \to X$ is said to be:\\
(i) {\it $M$-weakly compact} if $A(x_n) \longrightarrow 0$ in $X$ whenever $(x_n)_n$ is a bounded disjoint sequence in $E_1 \times \cdots \times E_m$ (with respect to the coordinatewise ordering in the cartesian product). \\
(ii) {\it Separately $M$-weakly compact} if $A$ is $M$-weakly compact in each variable, that is, for all $a_1 \in E_1, \ldots, a_m \in E_m$ and every $i \in \{1, \ldots, m\}$, the bounded linear operator $$x_i \in E_i \mapsto A(a_1, \ldots, a_{i-1}, x_i, a_{i+i}, \ldots, a_m) \in X$$ is $M$-weakly compact.
\end{definition}


It is clear that the sets of such operators are linear subspaces of ${\cal L}(E_1, \ldots, E_m;X)$. Let us see that these two classes are not related.

\begin{example} \label{o2ny}\rm Let $A \colon \ell_2 \times \ell_2 \to \ell_2$ be the bilinear operator from Example \ref{cont.Lfc-ex1}, which we proved to be separately $M$-weakly compact. Since the sequence $(e_n)_n$ of canonical unit vectors is bounded and disjoint in $\ell_2$ and $\|A(e_n,e_n)\| = \|e_n\| = 1$ for every $n$, we conclude that $A$ fails to be $M$-weakly compact.
\end{example}

Proposition \ref{pl4x} suggests that $M$-weakly compact multilinear operators are separately $M$-weakly compact. The fact that this is not true (in \cite[Example 3.2]{vinger} one can find $M$-weakly compact bilinear forms that are not separately $M$-weakly compact), is the first clue that this section does not go hand-in-hand with the previous section.


In Theorem \ref{Teo.M-frac.compBV} we shall prove, among many other things, that if the adjoint $A^*$ of a multilinear operator $A$ is $L$-weakly compact, then $A$ must be $M$-weakly compact and separately $M$-weakly compact. Thus, the two examples above show that these two properties are not enough for $A^*$ to be $L$-weakly compact. In summary, sufficient conditions on $A$ for $A^*$ to be $L$-weakly compact are yet to be found. In this search, it is a natural step to explore a few other multilinear transpositions of the class of $M$-weakly compact operators.

We start by considering an already studied different notion of disjointness in the cartesian product. For $x$ and $y$ in a Riesz space, the following is well known:\\
$\bullet$ $x \perp y$ implies $ax \perp b y$ for all $a,b \in \mathbb{R}$.\\
$\bullet$ $|y| \leq |x|$ implies $y \perp x$. \\
$\bullet$ If  $T \colon E \to F$ is a Riesz homomorphism and $x \perp y$, then  $T(x)\perp T(y)$.

The following concept, which is a particular case of a more general notion introduced by Kusraeva \cite{Kusraeva}, clearly fulfills the conditions of the first  two bullets above.

\begin{definition}\rm The vectors $(x_1,\dots,x_m)$, $(y_1,\dots,y_m)$ in $E_1\times\cdots\times E_m$ are said to be {\it sporadically disjoint} if $x_i \perp y_i$ in $E_i$ for some $i\in\{1,\dots,m\}$.
\end{definition}

Recall that an $m$-linear operator $A \colon E_1 \times \cdots \times E_m \to F$ is a {\it Riesz multimorphism}, or a {\it Riesz $m$-morphism}, if $|A(x_1, \ldots, x_m)| = A(|x_1|, \ldots, |x_m|)$ for every $(x_1, \ldots, x_m) \in E_1 \times \cdots \times E_m$. The next result, due to Boulabiar, proves that sporadic disjointness also fulfills the condition of the third bullet above.

\begin{proposition} {\rm \cite[Lemma 1.2]{Boulabiar}} \label{Bol.lemma.1.2}
 If $A\colon E_1\times\cdots\times E_m\to F$ is a Riesz $m$-morphism and $(x_1,\dots,x_m),(y_1,\dots,y_m)$ are sporadically disjoint in $E_1\times\cdots\times E_m$, then $A(x_1,\dots,x_m)\perp A(y_1,\dots,y_m)$ in $F$.
\end{proposition}

We shall need one more close connection of sporadic disjointness with the theory of multilinear operators on Banach lattices, namely, the fact that sporadic disjoint elements of the cartesian product correspond exactly to disjoint elementary tensors in the positive projective tensor product (cf. Lemma \ref{TensosseEspor}). Let $E_1\overline{\otimes}\cdots\overline{\otimes}E_m$ denote the Fremlin tensor product of the Archimedean Riesz spaces $E_1, \ldots, E_m$ (see \cite{FremlinTPAVL,FremlinTPBL,Buskes3,Buskes2}).  If $E_1, \ldots, E_m$ are Banach lattices, then the functional
$$x \in E_1\overline{\otimes}\cdots\overline{\otimes}E_m \mapsto \|x\|_{|\pi|}=\inf \left\{ \sum_{j=1}^k \|x_j^1\| \cdots \|x_j^m\| : x_j^i\in {E_i}^+,\, |x| \leq \sum_{j=1}^k x_j^1 \otimes \cdots \otimes x_j^m \right\}
$$
defines a lattice norm on $E_1 \overline{\otimes} \cdots \overline{\otimes} E_m$ (see \cite{FremlinTPBL}).  The completion of $E_1 \overline{\otimes} \cdots \overline{\otimes} E_m$ with respect to the norm $\|\cdot\|_{|\pi|}$, which is a Banach lattice \cite[Theorem 4.2]{Alip} containing $E_1 \otimes \cdots \otimes E_m$ as a dense subspace (see \cite[p.\,850, (g)]{Buskes3} or \cite[p.\,118]{Buskes}), is denoted by  $E_1{\hat{\otimes}}_{|\pi|}\cdots{\hat{\otimes}}_{|\pi|} E_m$ and called the {\it positive projective tensor product of} $E_1, \ldots, E_m$. Keeping the notation of the literature specialized in Banach lattices/Riesz spaces, by
$$\otimes\colon E_1\times\cdots\times E_m\to E_1{\hat{\otimes}}_{|\pi|}\cdots{\hat{\otimes}}_{|\pi|} E_m~,~\otimes(x_1,\dots,x_m)=x_1\otimes\cdots\otimes x_m,$$
we denote the canonical norm one Riesz $m$-morphism.

Let $E_1, \ldots, E_m,F$ be Banach lattices with $F$ Dedekind complete. For every operator $A \in \mathcal{L}_r(E_1,\dots,E_m;F)$ there exists a unique $A^\otimes \in \mathcal{L}_{r}(E_1{\hat{\otimes}}_{|\pi|}\cdots{\hat{\otimes}}_{|\pi|} E_m;F)$ such that $A = A^\otimes \circ \otimes$. Moreover, the correspondence
\begin{equation}A \in \mathcal{L}_r(E_1,\dots,E_m;F) \mapsto A^\otimes \in \mathcal{L}_{r}(E_1{\hat{\otimes}}_{|\pi|}\cdots{\hat{\otimes}}_{|\pi|} E_m;F) \label{km22}
\end{equation}
is an isometric isomorphism, where both of the spaces are endowed with their respective regular norms, and a Riesz isomorphism \cite[Proposition 3.3]{Buskes3}.

\begin{lemma} \label{TensosseEspor}
       Two vectors $(x_1,\dots,x_m)$, $(y_1,\dots,y_m)$ in $E_1\times\cdots\times E_m$ are sporadically disjoint if and only if $x_1\otimes\cdots\otimes x_m\perp y_1\otimes\cdots\otimes y_m$ in $E_1{\hat{\otimes}}_{|\pi|}\cdots{\hat{\otimes}}_{|\pi|} E_m$.
\end{lemma}

\begin{proof} Since the canonical $m$-linear operator $\otimes$ is a Riesz multimorphism, one direction follows from Proposition \ref{Bol.lemma.1.2}. Conversely, suppose that $x_1\otimes\cdots\otimes x_m\perp y_1\otimes\cdots\otimes y_m$ in $E_1{\hat{\otimes}}_{|\pi|}\cdots{\hat{\otimes}}_{|\pi|} E_m$. Using again that $\otimes$ is a Riesz $m$-morphism, $|x_1\otimes\cdots\otimes x_m|=|x_1|\otimes\cdots\otimes |x_m|$ holds true, so we can assume, wlog, that the vectors $x_1, y_1, \ldots, x_m,y_m $ are all positive.  For each $1\leq i\leq m$, let $0\leq z_i\in E_i$ be such that $0\leq z_i\leq x_i$ e $0\leq z_i\leq y_i$. Riesz multimorphism are Riesz homomorphisms separately, so $\otimes$ preserves disjointness fixing $(m-1)$ positive coordinates. It follows that
\begin{align*}
z_1\otimes x_2\cdots\otimes x_m&=(z_1\wedge x_1)\otimes x_2\cdots\otimes x_m=(z_1\otimes x_2\cdots\otimes x_m)\wedge (x_1\otimes x_2\cdots\otimes x_m).
\end{align*}
Hence, $z_1\otimes x_2\cdots\otimes x_m\leq x_1\otimes x_2\cdots\otimes x_m,$ and, analogously, $z_1\otimes y_2\cdots\otimes y_m\leq y_1\otimes y_2\cdots\otimes y_m.$
Repeating the same reasoning for the other variables, we get 
\begin{align*}
0\leq z_1\otimes\cdots\otimes z_m\leq x_1\otimes\cdots\otimes x_m\quad\text{and}\quad 0\leq z_1\otimes\cdots\otimes z_m\leq y_1\otimes\cdots\otimes y_m.
\end{align*}
As $x_1\otimes\cdots\otimes x_m\perp y_1\otimes\cdots\otimes y_m$,  $z_1\otimes\cdots\otimes z_m= 0$. Therefore, $\|z_1\| \cdots \|z_m\| = \| z_1\otimes\cdots\otimes z_m\|_{|\pi|} = 0$, from which we conclude that $z_i=0$ for some $1\leq i\leq m$, that is, $x_i\perp y_i$.
\end{proof}

At this point, the next definition is quite natural. Obviously, a sequence $(a_n)_n$ in  $E_1\times\cdots\times E_m$ is said to be {\it sporadically disjoint} if $a_n$ and $a_k$ are sporadically disjoint whenever $n \neq k$.

\begin{definition}\rm
An $m$-linear operator $A\in\mathcal{L}(E_1,\dots,E_m;X)$ is {\it sporadically $M$-weakly compact} if $A(a_k)\longrightarrow 0$ in $X$ for every sporadically disjoint sequence $(a_n)_n$ in $B_{E_1\times\cdots\times E_m}$.
\end{definition}

In order to define one more class of $M$-weakly compact-type multilinear operators, let us fix the following notation: Let $A\colon V_1\times\cdots\times V_m\longrightarrow W$ be an $m$-linear operator between linear spaces. Fixed $p_1<\cdots<p_k\in\{1,\dots,m\}$ and  $a_{p_1}\in V_{p_1},\dots,a_{p_k}\in V_{p_k}$, it is clear that the map  $A_{a_{p_1},\dots,a_{p_k}}\colon V_1\times\stackrel{[p_1,\dots,p_k]}{\dots}\times V_m\longrightarrow Y$ given by \[A_{a_{p_1},\dots,a_{p_k}}(x_1,\stackrel{[p_1,\dots,p_k]}{\dots},x_m)=A(x_1,\dots,x_{p_1-1},a_{p_1},x_{p_1+1},\dots,x_{p_k-1},a_{p_k},x_{p_k+1},\dots,x_m),\]
  where $\stackrel{[p_1,\dots,p_k]}{\dots}$ means that the coordinates  $p_1,\dots,p_k$ have been omitted, is an $(m-k)$-linear operator. 
We say that  $A_{a_{p_1},\dots,a_{p_k}}$ is the {\it resulting operator letting} $a_{p_1}\in X_{p_1},\dots,a_{p_k}\in X_{p_k}$ {\it fixed}. The next concept was introduced and proved to be useful in \cite{vinger}.

\begin{definition}\rm \cite[Definition 3.3]{vinger}\label{defseparM}
An $m$-linear operator $A\in\mathcal{L}(E_1,\dots,E_m;X)$ is said to be {\it strongly $M$-weakly compact} if, regardless of the numbers $0\leq k\leq m-1, p_1<\cdots<p_k\in\{1,\dots,m\}$ and the vectors $a_{p_1}\in E_{p_1},\dots,a_{p_k}\in E_{p_k}$, the resulting $(m-k)$-linear operator $A_{a_{p_1},\dots,a_{p_k}}$ letting $a_{p_1}, \ldots, a_{p_k}$ fixed is $M$-weakly compact.
\end{definition}

Despite of seeming to be a quite small class, plenty of examples of strongly $M$-weakly compact operators are provided in \cite{vinger}. Reinforcing that this class is not small at all, next we add some more examples, showing how $M$-weakly compact linear operators can be used to produce strongly $M$-weakly compact multilinear operators.

\begin{proposition}\label{hxm7}
Let $E_1,\dots,E_m$ be Banach lattices, let $Y_1,\dots,Y_m,X$ be Banach spaces, let $U_1\colon E_1\to Y_1, \ldots, U_m\colon E_m\to Y_m$  be $M$-weakly compact linear operators, and let $B\colon Y_1 \times \cdots \times Y_m \to Y$ be a continuous $m$-linear operator. Then $B \circ (U_1, \ldots,  U_m)$ is strongly $M$-weakly compact, that is, the map $A\colon E_1\times\cdots\times E_m\to X$ given by  $$A(x_1,\dots,x_m)=B(U_1(x_1),\dots,U_m(x_m)),$$
is a strongly $M$-weakly compact $m$-linear operator.
\end{proposition}

\begin{proof} It is clear that $A$ is $m$-linear and continuous. Let $A_{a_{p_1},\dots,a_{p_k}}$ be the resulting $(m-k)$-linear operator letting  $a_{p_1}\in E_{p_1},\dots,a_{p_k}\in E_{p_k}$ fixed, where $k \in \{0,1,\ldots, m-1\}$ and $p_1 < \cdots < p_k \in \{1, \ldots, m\}$. Given a  bounded disjoint sequence $((x_1^n,\stackrel{[p_1,\dots,p_k]}{\dots},x_m^n))^n$ in $E_1\times\stackrel{[p_1,\dots,p_k]}{\dots}\times E_m$, for every $1\leq i\leq m$ with $i\neq p_1,\dots,i \neq p_k$, $(x_i^n)_n$ is a bounded disjoint sequence in $E_i$, hence $U_i(x_i^n)\longrightarrow 0$ because $U_i$ is $M$-weakly compact. Therefore,
\begin{align*}
\|A_{a_{p_1},\dots,a_{p_k}}&(x_1^n,\stackrel{[p_1,\dots,p_k]}{\dots},x_m^n)\|\\
&=\|A(x_1^n,\dots,x_{p_1-1}^n,a_{p_1},x_{p_1+1}^n,\dots,x_{p_k-1}^n,a_{p_k},x_{p_k+1}^n,\dots,x_m^n)\|\\
&=\|B(U_1(x_1^n)\dots,U_{p_1-1}(x_{p_1-1}^n),U_{p_1}(a_{p_1}),U_{p_1+1}(x_{p_1+1}^n),\\
&~~~~\dots,U_{p_k-1}(x_{p_k-1}^n),U_{p_k}(a_{p_k}),U_{p_k+1}(x_{p_k+1}^n),\dots,U_m(x_m^n))\|\\
&\leq \|B\|\cdot\|U_1(x_1^n)\|\stackrel{[p_1,\dots,p_k]}{\cdots}\|U_m(x_m^n))\|\cdot\|U_{p_1}(a_{p_1}))\|\cdots\|U_{p_k}(a_{p_k}))\|\longrightarrow 0,
\end{align*}
proving that $A_{a_{p_1},\dots,a_{p_k}}$ is $M$-weakly compact.
\end{proof}

\noindent{\bf Problem 1.} Let $E_1, \ldots, E_m$ be Riesz spaces. If the sequence $(x^1_j, \ldots, x^m_j)_j$ is sporadically disjoint in $E_1 \times \cdots \times E_m$, then there exists $i \in \{1, \ldots, m\}$ such that the sequence $(x^i_j)_j$ admits a disjoint subsequence.

\medskip

Very soon we will prove that sporadically $M$-weakly compact operators are strongly $M$-weakly compact. If the answer to Problem 1 turns out to be affirmative, then the following improvement of Proposition \ref{hxm7} holds.

\medskip

\noindent{\bf Problem 2.} Multilinear operators as those in Proposition \ref{hxm7} are sporadically $M$-weakly compact.

\medskip

Next we prove the main result of this section, where we characterize the multilinear operators whose adjoints are $L$-weakly compact and relate this class with the other classes of operators of $M$-weakly compact-type.


\begin{theorem}\label{Teo.M-frac.compBV} Let 
$A \colon E_1 \times \cdots \times E_m \to F$ be an $m$-linear operator of order bounded variation. Consider the following statements. \\
{\rm (i)} $A^*\colon F^*\to \mathcal{L}_{r}(E_1,\dots,E_m;\R)$ is an $L$-weakly compact linear operator.\\
{\rm (ii)} $A^\otimes\colon E_1{\hat{\otimes}}_{|\pi|}\cdots{\hat{\otimes}}_{|\pi|} E_m\to F$ is an $M$-weakly compact linear operator.\\
{\rm (iii)} There are a Banach lattice $G$, a Riesz $m$-morphism $B\colon E_1\times\cdots\times E_m\to G$, and an order bounded $M$-weakly  compact linear operator $T\colon G\to F$ such that $A=T\circ B$.\\
{\rm (iv)} $\|A(x_1^n,\ldots, x_m^n)\|\longrightarrow 0$ for every bounded disjoint sequence of elementary tensors $(x_1^n\otimes\cdots\otimes x_m^n)_n$ in $E_1{\hat{\otimes}}_{|\pi|}\cdots{\hat{\otimes}}_{|\pi|} E_m$.\\
{\rm (v)} $A$ is sporadically $M$-weakly compact.\\
{\rm (vi)} $A$ is strongly $M$-weakly compact.\\
{\rm (vii)} $A$ is separately $M$-weakly compact.\\
{\rm (viii)} $A$ is $M$-weakly compact.

Then {\rm (i)}$\Leftrightarrow${\rm (ii)}$\Leftrightarrow${\rm (iii)}$\Rightarrow${\rm (iv)}$\Leftrightarrow${\rm (v)}$\Rightarrow${\rm (vi)}$\Rightarrow${\rm [(vii)} and {\rm(viii)]}. Moreover, {\rm(vii)}$\nRightarrow${\rm(viii)}, {\rm(viii)}$\nRightarrow${\rm(vii)}, {\rm(viii)}$\nRightarrow${\rm(v)},  {\rm[(vii) and (viii)]}$\nRightarrow${\rm(vi)}.
\end{theorem}

\begin{proof} First note that $A$ is continuous by Proposition \ref{obv-cont}, hence its adjoint $A^*$ is also continuous (for the linear case, see \cite[Proposition 4.3.11]{Botelho}, the multilinear case is analogous) and its linearization  $A^\otimes$ is continuous as well because $A$ has order bounded variation \cite[Theorem 3.1]{Buskes2}. 

{\rm (i)}$\Leftrightarrow${\rm (ii)} Applying (\ref{km22}) to the case of $m$-linear forms, the map
$$I\colon \mathcal{L}_r(E_1,\dots,E_m;\R)\to(E_1{\hat{\otimes}}_{|\pi|}\cdots{\hat{\otimes}}_{|\pi|} E_m)^* ~,~ I(A)=A^\otimes,$$
is an isometric isomorphism and a Riesz isomorphism. Denoting $U: = B_{E_1\times\cdots\times E_m}$ and $V :=B_{E_1{\hat{\otimes}}_{|\pi|}\cdots{\hat{\otimes}}_{|\pi|} E_m}$, the fact that $I$ is a Riesz homomorphism gives ${\rm Sol}(I(A^*(B_{X^*})))= I({\rm Sol}(A^*(B_{X^*})))$. Since Riesz homomorphisms preserve disjointness and $I$ is an isometric isomorphism, the two conditions below are equivalent: 
\begin{enumerate}
\item[(a)] $\|\phi_n\|_r=\sup\limits_{u\in U}||\phi_n|(u)|\longrightarrow 0$~~ for each disjoint sequence $(\phi_n)_n$ in ${\rm Sol}(A^*(B_{X^*}))$.
\item[(b)] $\|\Phi_n\|=\sup\limits_{v\in V}|\Phi_n(v)|\longrightarrow 0$~\, for each disjoint sequence $(\Phi_n)_n$ in ${\rm Sol}(I(A^*(B_{X^*})))= $ $I({\rm Sol}(A^*(B_{X^*})))$.
\end{enumerate}
Calling on the Burkinshaw–Dodds Theorem once again, condition (b) above is equivalent to
\begin{align*}
\sup\limits_{\Phi^\in I(A^*(B_{X^*}))}|\Phi(v_n)|&=\sup_{x^*\in B_{X^*}}|[I(A^*(x^*))](v_n)|=\sup_{x^*\in B_{X^*}}|[(A^\otimes)^*(x^*)](v_n)|\\&=\sup_{x^*\in B_{X^*}}|x^*(A^\otimes(v_n))|=\|A^\otimes(v_n)\|\longrightarrow0
\end{align*}
for every disjoint sequence $(v_n)_n$ in ${\rm Sol}(V)=V$. It follows that $A^*$ is $L$-weakly compact if and only if $A^\otimes$ is $M$-weakly compact.

(ii)$\Rightarrow$(iii) Just take $F = E_1{\hat{\otimes}}_{|\pi|}\cdots{\hat{\otimes}}_{|\pi|} E_m$, $T = A^\otimes$ and note that $A^\otimes$ is order bounded because $A$ has order bounded variation \cite[Theorem 3.1]{Buskes2}.

(iii)$\Rightarrow$(ii) For any elementary tensor $x_1\otimes\cdots\otimes x_m$ in $E_1{\hat{\otimes}}_{|\pi|}\cdots{\hat{\otimes}}_{|\pi|} E_m$,  $$A^\otimes(x_1\otimes\cdots\otimes x_m)=A(x_1,\dots,x_m)=T(B(x_1,\dots,x_m))=T(B^\otimes(x_1\otimes\cdots\otimes x_m)).$$
By linearity, $A^\otimes(z)=T(B^\otimes(z))$ for every $z\in E_1\otimes\cdots\otimes E_m$. As $E_1\otimes\cdots\otimes E_m$ is dense in $E_1{\hat{\otimes}}_{|\pi|}\cdots{\hat{\otimes}}_{|\pi|} E_m$, by continuity we have $A^\otimes=T\circ B^\otimes$. Given a bounded disjoint sequence $(z_n)_n$ in $E_1{\hat{\otimes}}_{|\pi|}\cdots{\hat{\otimes}}_{|\pi|} E_m$, $(B^\otimes(z_n))_n$ is a bounded disjoint sequence in $F$ because $B^\otimes$ is a Riesz homomorphism by \cite[Theorem 4.2(i)]{FremlinTPAVL}. Since $T$ is $M$-weakly compact, $A^\otimes(z_n)=T(B^\otimes(z_n))\longrightarrow 0$, which proves (ii).

(ii)$\Rightarrow$(iv) This implication is obvious.

(iv)$\Leftrightarrow$(v)  This equivalence follows from Lemma \ref{TensosseEspor}.

(v)$\Rightarrow$(vi) Let $k \in \{0, 1, \ldots, m-1\}, p_1<\cdots<p_k\in\{1,\dots,m\}$ and $a_{p_1}\in E_{p_1},\dots,a_{p_k}\in E_{p_k}$ be given. For bounded disjoint sequences $(x_i^n)_n$ in $E_i$, $1\leq i\leq m$, it is plain that the sequence  $$(x_{a_{p_1},\dots,a_{p_k}}^n)_n: =((x_1^n,\dots,x_{p_1-1}^n,a_{p_1},x_{p_1+1}^n,\dots,x_{p_k-1}^n,a_{p_k},x_{p_k+1}^n,\dots,x_m^n))_n$$ is bounded and sporadically disjoint in $E_1\times\cdots\times E_m$. Since $A$ is sporadically $M$-weakly compact,  $A(x_{a_{p_1},\dots,a_{p_k}}^n)\longrightarrow 0$, therefore $A$ is strongly $M$-weakly compact.

(vi)$\Rightarrow$(vii) This implication follows by  taking $k = m-1$ in Definition  \ref{defseparM}, and (vi)$\Rightarrow$(viii) follows by taking $k = 0$ in the same definition.

{\rm(vii)}$\nRightarrow${\rm(viii)} See Example \ref{o2ny}.

{\rm(viii)}$\nRightarrow${\rm(vii)} See \cite[Example 3.2]{vinger}.

{\rm(viii)}$\nRightarrow${\rm(v)} It follows from {\rm(viii)}$\nRightarrow${\rm(vii)} and (v)$\Rightarrow$(vi)$\Rightarrow$(vii).

{\rm[(vii) and (viii)]}$\nRightarrow${\rm(vi)} Let $A\colon \ell_2\times \ell_2\to \ell_2$ be the positive bilinear operator  we proved in Example \ref{cont.Lfc-ex1} to be, among other things, separately $M$-weakly compact. Let $E$ be any Banach lattice such that $E^*$ has order continuous norm and fix $0 < x^* \in E^*$. 
Consider the positive $3$-linear operator
$$B\colon E\times \ell_2\times\ell_2\to\ell_2~,~ B(x,a,b)=x^*(x)\!\cdot\! A(a,b).$$
As positive multilinear operators have order bounded variation \cite[p.\,49]{Buskes2}, $B \in {\cal L}_{bv}(E,\ell_2, \ell_2; \ell_2)$. Fixing the second and the third variables, the resulting linear operator is a multiple of $x^*$, which is $M$-weakly compact by \cite[Theorem 4.69]{Alip}. Fixing any two other coordinates, the resulting operator is a multiple of a linear operator obtained by fixing one coordinate of $A$, which is $M$-weakly compact because $A$ is separately $M$-weakly compact. Thus, $B$ is separately $M$-weakly compact. If $(x_n)_n$ is a disjoint sequence in $B_E$ and $(\alpha^n)_n,(\beta^n)_n$ are disjoint sequences $B_{\ell_2}$, then $$\|B(x_n,\alpha^n,\beta^n)\|_2=|x^*(x_n)|\!\cdot\!\|A(\alpha^n,\beta^n)\|_2\leq
|x^*(x_n)|\!\cdot\!\|A\|\!\cdot\!\|\alpha^n\|_2
\!\cdot\!\|\beta^n\|_2\longrightarrow 0,$$ because $x^*$ is $M$-weakly compact. This shows that $B$ is $M$-weakly compact. To complete the proof, take $x\in E$ with $x^*(x)=1$ and note that the resulting bilinear operator fixing $x$ in the first variable of $B$ is equal to  $A$, which is not $M$-weakly compact by Example \ref{o2ny}. Therefore, $B$ fails to be strongly $M$-weakly compact.
\end{proof}

\begin{remark}\rm (a) In the theorem above, the following implications hold, with the same proofs, for a continuous $m$-linear operator taking values in a Banach space: (iii)$\Rightarrow$(iv)$\Leftrightarrow${\rm (v)}$\Rightarrow${\rm (vi)}$\Rightarrow${\rm [(vii)} and {\rm(viii)]}. The counterexamples also apply to this situation,  because all of them are positive multilinear operators between Banach lattices, hence continuous.\\
(b) As to condition (iii) above, it is worth noting that there exist non-order bounded $M$-weakly compact linear operators (see \cite[Theorem 2.1]{chenwickstead} and its proof).
\end{remark}

\noindent{\bf Acknowledgment.} The authors thank Luis Alberto Garcia for his helpful remarks.

\bigskip
\noindent Geraldo Botelho~~~~~~~~~~~~~~~~~~~~~~~~~~~~~~~~~~~~~~Ariel Mon\c c\~ao\\
Instituto de Matem\'atica e Estat\'istica~~~~~~~~~\,\,\,Departamento de Matemática\\
Universidade Federal de Uberl\^andia~~~~~~~~~~~~~Universidade Federal de Minas Gerais\\
38.400-902 -- Uberl\^andia -- Brazil~~~~~~~~~~~~~~~~~31.270-901 -- Belo Horizonte -- Brazil\\
e-mail: botelho@ufu.br~~~~~~~~~~~~~~~~~~~~~~~~~\,~~~~~e-mail: arieldeom@hotmail.com

\begin{thebibliography}{99}\small
\vspace*{-0.4em}
\bibitem{invitation} Abramovich, Y. A., Aliprantis, C. D., {\it An Invitation to Operator Theory}, Graduate Studies in Mathematics, 51, American Mathematical Society, Providence, 2002.

\vspace*{-0.4em}
\bibitem{Alip} Aliprantis, C. D.,  Burkinshaw, O., {\it Positive Operators}, Springer, Dordrecht, 2006.

\vspace*{-0.4em}
\bibitem{Aron} Aron, R. M., Schottenloher, M., {\it Compact holomorphic mappings on Banach spaces and the approximation property}. J. Functional Analysis {\bf 21} (1976), no. 1, 7-30.

\vspace*{-0.4em}
\bibitem{Torres} Bényi, Á., Li, G., Oh, T., Torres, R. H., {\it Compact bilinear operators and paraproducts revisited}, Canad. Math. Bull. {\bf 68}(1) (2025), 44-59.

\vspace*{-0.4em}
\bibitem{khazhak} Botelho, G., Bu, Q., Ji, D., Navoyan, K., {\it The positive Schur property on positive projective tensor products and spaces of regular multilinear operators}, Monatsh. Math. {\bf 197} (2022), 565-578.

\vspace*{-0.4em}
\bibitem{vinger} Botelho, G., Miranda, V. C. C., {\it Compact positive multilinear operators on Banach lattices}, arXiv:2509.04652v1, 2025.


\vspace*{-0.4em}
\bibitem{prims} Botelho, G., Pellegrino, D., Rueda R., {\it On composition ideals of multilinear mappings and homogeneous polynomials}, Publ.
Res. Inst. Math. Sci. {\bf 43} (2007) 1139–1155.




\vspace*{-0.4em}
\bibitem{Botelho} Botelho, G., Pellegrino, D., Teixeira, E., {\it Introduction to Functional Analysis}, Springer, Cham, 2025.

\vspace*{-0.4em}
\bibitem{leodanmathscand} Botelho, G., Torres, L. A., {\it Generalized adjoints and applications to composition operators}, Math. Scand. {\bf 126} (2020), 367-386.


\vspace*{-0.4em}
\bibitem{Boulabiar} Boulabiar, K., {\it Some aspects of Riesz multimorphisms}, Indag. Math. (N.S.) {\bf 13}(4) (2002), 419-432.

\vspace*{-0.4em}
\bibitem{Boulabiar2} Boulabiar, K., Buskes, G., Page, R., {\it On some properties of bilinear maps of order bounded variation}, Positivity {\bf 9} (2005), 401–414.



\vspace*{-0.4em}
\bibitem{Buskes3} Bu, Q., Buskes, G., {\it Polynomials on Banach lattices and positive tensor products},
J. Math. Anal. Appl. {\bf 388} (2012), 845-862.

\vspace*{-0.4em}
\bibitem{Buskes} Bu, Q., Buskes, G., Kusraev, A. G., {\it Bilinear maps on products of vector lattices: a survey}, Positivity, Trends Math., Birkh\"auser, Basel, 97-126, 2007.

\vspace*{-0.4em}
\bibitem{Burk-Dodds} Burkinshaw O., Dodds P., {\it Disjoint sequences, compactness, and semireflexivity in locally convex Riesz spaces}, Illinois J. Math. {\bf 21} (1977), no. 4, 759-775.

\vspace*{-0.4em}
\bibitem{Buskes2} Buskes, G., van Rooij, A., {\it Bounded variation and tensor products of Banach lattices},  Positivity {\bf 7} (2003), no. 1-2, 47–59.

\vspace*{-0.4em}
\bibitem{chenwickstead} Chen, Z. L., Wickstead, A. W., {\it L-weakly and M-weakly compact operators}, Indag. Math. (N.S.) {\bf 10} (1999) no. 3, 321-336.

\vspace*{-0.4em}
\bibitem{erhanpilar} \c Caliskan, E., Rueda, P., {\it Compactness and s-numbers for polynomials}, Atti Accad. Naz. Lincei Rend. Lincei Mat. Appl. {\bf 29} (2018), no. 1, 93–107.

\vspace*{-0.4em}
\bibitem{df} Defant, A., Floret, K., {\it Tensor Norms and Operator Ideals}, North-Holland, 1993.

\vspace*{-0.4em}
\bibitem{dineen} Dineen, S., {\it Complex Analysis on Infinite Dimensional Spaces}, Springer, 1999.

\vspace*{-0.4em}
\bibitem{FremlinTPAVL} Fremlin, D. H., {\it Tensor products of Archimedean vector lattices}, Amer. J. Math. {\bf 94} (1972), no. 3, 777-798.

\vspace*{-0.4em}
\bibitem{FremlinTPBL} Fremlin, D. H., {\it Tensor products of Banach lattices}, Math. Ann. 211 (1974), 87-106.

\vspace*{-0.4em}
\bibitem{smooth} H\'ajek, P., Johanis, M., {\it Smooth Analysis in Banach Spaces}, De Gruyter Series in Nonlinear Analysis and Applications, 19, De Gruyter, Berlin, 2014. 



\vspace*{-0.4em}
\bibitem{Kusraeva} Kusraeva, Z. A. {\it Sums of disjointness preserving multilinear operators}, Positivity {\bf 25} (2021), 669–678. 

\vspace*{-0.4em}
\bibitem{Kusraeva2} Kusraeva, Z. A. {\it Sums of powers of orthogonally additive polynomials}, J. Math. Anal. Appl. {\bf 519} (2023), no. 2, Paper No. 126766, 14 pp.

\vspace*{-0.4em}
\bibitem{silviapablo} Lassalle, S., and Turco, P., {\it Polynomials and holomorphic functions on $\cal A$-compact sets in
Banach spaces}, J. Math. Anal. Appl. {\bf 463} (2018), no. 2, 1092–1108.

\vspace*{-0.4em}
\bibitem{Loane} Loane, J., {\it Polynomials on Riesz Spaces}, PhD. Thesis, Department of Mathematics, National University of Ireland, Galway, 2007.

\vspace*{-0.4em}
\bibitem{MeyerPeter} Meyer-Nieberg, P., {\it Banach Lattices}, Springer-Verlag, 1991. 


\vspace*{-0.4em}
\bibitem{Mujicatrans} Mujica, J., {\it Linearization of bounded holomorphic mappings on Banach spaces}, Trans. Amer. Math. Soc. {\bf 324} (1991) 867–887.

\vspace*{-0.4em}
\bibitem{mujica} Mujica, J., {\it Complex Analysis in Banach Spaces}, Dover Publ., 2010.


\vspace*{-0.4em}
\bibitem{Ryanpacific} Ryan, R. A., {\it Weakly compact holomorphic mappings on Banach spaces}, Pacific J. Math. {\bf 131} (1988) 179–190.

\vspace*{-0.4em}
\bibitem{Ryan} Ryan, R. A., {\it Introduction to Tensor Products of Banach Spaces}, Springer, 2002.

\vspace*{-0.4em}
\bibitem{Schaefer} Schaefer, H. H., {\it Banach Lattices and Positive Operators}, Springer Berlin, Heidelberg, 1974.

\vspace*{-0.4em}
\bibitem{pablo} Turco, P., {\it $\cal A$-compact mappings}, Rev. R. Acad. Cienc. Exactas Fís. Nat. Ser.A Mat. RACSAM {\bf 110} (2016), no. 2, 863–880.

\vspace*{-0.4em}
\bibitem{Wickstead} Wickstead, A. W., {\it Order bounded operators may be far from regular}, In: Abramovich, Y., Avgerinos, E., Yannelis, N.C. (eds) Functional Analysis and Economic Theory, pp. 109-118, Springer, Berlin, 1998.



%
%
%
%
%
%
%
%
%
%
%
%
%
%
%
%
%
%
%
%
%
%
%
%
%
%
%
%
%
%
%
%
%
%
\end{thebibliography}
\end{document}